\documentclass[11pt,oneside,letterpaper]{article}
\usepackage[dvips,final]{graphicx}
\usepackage[dvips]{geometry}
\usepackage{setspace}
\usepackage{mathptmx}
\usepackage{times}
\usepackage{amsmath}
\usepackage{amsthm}
\usepackage{amssymb}
\usepackage{amsfonts}
\usepackage{mathrsfs}
\usepackage[all,cmtip]{xy} 
\usepackage{enumitem}
\usepackage{verbatim} 
\usepackage[usenames,dvipsnames,svgnames,table]{xcolor}
\usepackage{flafter}
\usepackage{hyperref}
\usepackage{tocloft}

\DeclareFontFamily{U}{wncy}{}
\DeclareFontShape{U}{wncy}{m}{n}{<->wncyr10}{}
\DeclareSymbolFont{mcy}{U}{wncy}{m}{n}
\DeclareMathSymbol{\Sh}{\mathord}{mcy}{"58} 
\newtheorem{theorem}{Theorem}[section] 
\newtheorem{lemma}[theorem]{Lemma}

\newtheorem{proposition}[theorem]{Proposition}
\theoremstyle{definition}

\newtheorem{remark}[theorem]{Remark}

\newtheorem*{rep@theorem}{\rep@title}
\newcommand{\newreptheorem}[2]{%
\newenvironment{rep#1}[1]{%
 \def\rep@title{#2 \ref{##1}}%
 \begin{rep@theorem}}%
 {\end{rep@theorem}}}
\newreptheorem{theorem}{Theorem}

\newcommand{\Sel}{\mathrm{Sel}}

\newcommand{\rank}{\mathrm{rank}}
\newcommand{\Gal}{\mathrm{Gal}}

\newcommand{\cor}{\mathrm{cor}} 
\newcommand{\Tr}{\mathrm{Tr}} 
\newcommand{\Hom}{\mathrm{Hom}} 
\newcommand{\Frob}{\mathrm{Frob}}
\newcommand{\Ker}{\mathrm{Ker}}

\newcommand{\Aut}{\mathrm{Aut}}

\newcommand{\res}{\mathrm{res}}

\font\smallit=cmti10

\begin{document}

\begin{center}
\uppercase{\bf Kolyvagin's method for Chow groups \\ of Kuga-Sato varieties over ring class fields}
\vskip 20pt
{\bf Yara Elias\footnote{Supported by a doctoral scholarship of the Fonds Qu\'eb\'ecois de la Recherche sur la Nature et les Technologies.}}\\
{\smallit Department of Mathematics, McGill University}\\
{\tt yara.elias@mail.mcgill.ca}\\

\end{center}

\begin{abstract}
We use an Euler system of Heegner cycles to bound the Selmer group associated to a modular form of higher even weight twisted by a ring class field character. This is an extension of Nekovar's result \cite{nekovar1992kolyvagin} that uses Bertolini and Darmon's refinement of Kolyvagin's ideas, as described in \cite{bertolini1990descent}.
\end{abstract}

\textbf{Acknowledgments.} I am very grateful to my advisor Henri Darmon who suggested this beautiful problem and guided me through its resolution and whose feedback was crucial to the writing of this paper. 
I thank Olivier Fouquet, Eyal Goren for many useful corrections and suggestions.

\section{Introduction}

Kolyvagin \cite{kolyvagin1990grothendieck} and \cite{gross1991kolyvagin} used an Euler system to bound the size of the Selmer group of certain elliptic curves over imaginary quadratic fields assuming the non-vanishing of a suitable Heegner point.
This implied that they have rank 1, and that their associated Tate-Shafarevich group is finite.
Combined with results of Gross and Zagier \cite{gross1986heegner}, this proved the Birch and Swinnerton-Dyer conjecture for analytic rank at most 1.
Bertolini and Darmon later adapted Kolyvagin's descent to Mordell-Weil groups over ring class fields \cite{bertolini1990descent}.  
Nekovar \cite{nekovar1992kolyvagin} then applied the methods of Euler systems to modular forms of even weight larger than 2 to describe the image by the Abel-Jacobi map of certain algebraic cycles of the associated Kuga-Sato varieties. 
Combined with results of Gross-Zagier \cite{gross1986heegner} and Brylinski \cite{brylinski1989}, this provided further grounds to believe the Bloch-Beilinson conjecture which generalizes Birch and Swinnerton-Dyer's.
The present work adapts ideas and techniques from the aforementionned articles to bound the size of Selmer groups associated to modular forms of even weight larger than 2 twisted by ring class characters of imaginary quadratic fields. 

Let $f$ be a normalized newform of level $\Gamma_0(N)$ and trivial nebentype where $N \geq 5$ and of even weight $2r > 2$
and let $$K=\mathbb{Q}(\sqrt{-D})$$ be an imaginary quadratic field satisfying the \emph{Heegner hypothesis} relative to $N$, that is, rational primes dividing $N$ split in $K$. 
For simplicity, we assume that $|\mathcal{O}_K^{\times}|=2$.
We fix a prime $p$ not dividing $ND\phi(N).$
Let $H$ be the ring class field of $K$ of conductor $c$ with $(c,NDp)=1$, and let $e$ be the exponent of $\Gal(H/K)$.
Let $F=\mathbb{Q}(a_1, a_2, \cdots, \mu_e)$ be the field generated over $\mathbb{Q}$ by the coefficients of $f$ and the $e$-th roots of unity $\mu_e$.
We denote by $A$ the $p$-adic \'{e}tale realization of the  motive associated to $f$ by Scholl \cite{scholl1990motives} and Deligne \cite{deligne1973modular} twisted by $r$ (see Section \ref{motive} for more details).
It will be viewed (by extending scalars appropriately) as a free $\mathcal{O}_F \otimes \mathbb{Z}_p$ module of rank 2, equipped with a continuous $\mathcal{O}_F$-linear action of $\Gal(\overline{\mathbb{Q}}/\mathbb{Q})$. Let $A_{\wp}$ be the localization of $A$ at a prime $\wp$ of $\mathcal{O}_F$ dividing $p$ as in Expression \ref{localized A}. Then $A_{\wp}$ is a free module of rank 2 over $\mathcal{O}_{\wp},$ the completion of $\mathcal{O}_F$ at $\wp$. 
The Selmer group $$S \subseteq H^1(H, A_{\wp}/p)$$ consists of the cohomology classes $c$ whose localizations $c_v$ at a prime $v$ of $H$ lie in 
$$\left\{
    \begin{array}{ll}
         H^1(H^{ur}_{v}/H_{v},A_{\wp}/p) \hbox{ for } v \hbox{ not dividing }Np\\
         H^1_f(H_v, A_{\wp}/p ) \hbox{ for } v \hbox{ dividing }p
    \end{array}
\right.$$
where $H^1_f(H_v, A_{\wp}/p)$ is the \emph{finite part} of $H^1(H_v, A_{\wp}/p)$ as in \cite{block1990lfunction}.
In our setting, since $A_{\wp}$ has good reduction at $p$, $H^1_f(H_v, M)=H^1_{cris}(H_v, M).$
Note that the assumptions we make will ensure that $H^1(H^{ur}_{v}/H_{v},A_{\wp}/p)=0$ for $v$ dividing $N$. 
The Galois group $$G=\Gal(H/K)$$ acts on $H^1(H,A_{\wp}/p)$ and preserves the unramified and cristalline classes, hence it acts on $S$. Assume that $p$ does not divide $|G|$.
We denote by $\hat{G}=\mathrm{Hom}(G, \mu_e)$ the group of characters of $G$
and by $$e_{\chi}=\dfrac{1}{|G|}\sum_{g \in G} \chi^{-1}(g) g$$ the projector onto the $\chi$-eigenspace given a character $\chi$ of $\hat{G}$.

By the Heegner hypothesis, there is an ideal $\mathcal{N}$ of $\mathcal{O}_{c}$, the order of $K$ of conductor $c$, such that $$\mathcal{O}_{c}/\mathcal{N} = \mathbb{Z}/N \mathbb{Z}.$$ Therefore, $\mathbb{C}/\mathcal{O}_{c}$ and $\mathbb{C}/\mathcal{N}^{-1}$ define elliptic curves related by an $N$-isogeny. As points of $X_0(N)$ correspond to elliptic curves related by $N$-isogenies, this provides a \emph{Heegner point} $x_{1}$ of $X_0(N)$. By the theory of complex multiplication, $x_1$ is defined over $H$.
Let $E$ be the corresponding elliptic curve. Then $E$ has complex multiplication by $\mathcal{O}_{c}$.
The Heegner cycle of conductor $c$ is defined as $$e_r (\mathrm{graph}(\sqrt{-D}))^{r-1}$$ for some appropriate projector $e_r$, (see Section \ref{motive} for more details).  
Let $\delta$ be the image by the $p$-adic \'{e}tale Abel-Jacobi map of the Heegner cycle of conductor $c$ viewed as an element of $H^1(H,A_{\wp}/p)$. 
We denote by $Fr(v)$ the arithmetic Frobenius element generating $\Gal(H_v^{ur}/H_v),$ and by $$I_v=\Gal(\overline{H_v}/H_v^{ur}).$$
This article is dedicated to the proof of the following statement:
\begin{theorem} \label{theorem 1}
Assume that $p$ is such that $$ \Gal \left( \mathbb{Q}(A_{\wp}/p)/\mathbb{Q} \right) \simeq \mathrm{GL}_2(\mathcal{O}_{\wp}/p), \ \ \ \ (p,ND\phi(N))=1, \ \hbox{ and } \ p \nmid |G|.$$
Suppose further that the eigenvalues of $Fr(v)$ acting on $A_{\wp}^{I_v}$ are not equal to 1 modulo $p$ for $v$ dividing $N$. 
Let $\chi \in \hat{G}$ be such that $e_{\overline{\chi}} \delta$ is not divisible by $p$.
Then the $\chi$-eigenspace $S^{\chi}$ of the Selmer group  $S$ has rank 1 over $\mathcal{O}_{\wp}/p$.
\end{theorem}

To prove Theorem \ref{theorem 1}, we first view the $p$-adic \'{e}tale realization $A$ of the twisted motive associated to $f$ in the middle \'{e}tale cohomology of the associated Kuga-Sato varieties. The main two ingredients of the proof are the refinement of an Euler system of so-called Heegner cycles first considered by Nekov{\'a}{\v{r}} and Kolyvagin's descent machinery adapted by Nekov{\'a}{\v{r}} \cite{nekovar1992kolyvagin} to the setting of modular forms. In order to bound the rank of the $\chi$-eigenspace of the Selmer group $S^{\chi}$, we use Local Tate duality and the local reciprocity law to obtain information on the local elements of the Selmer group. Using a global pairing of the Selmer group and Cebotarev's density theorem, we translate this local information about the elements of $S^{\chi}$ into global information. 

The main novelty is the adaptation of the techniques by Bertolini and Darmon in \cite{bertolini1990descent} to the setting of modular forms that allow us to get around the action of complex conjugation. Indeed, unlike the case where $\chi$ is trivial, the complex conjugation $\tau$ does not act on $S^{\chi}$ as it maps it to $S^{\overline{\chi}}.$

\section{Motive associated to a modular form} \label{motive}

In this section, we describe the $p$-adic \'{e}tale realization $A$ of the motive associated to $f$ by Scholl \cite{scholl1990motives} and Deligne \cite{deligne1973modular} twisted by $r$.
Consider the congruence subgroup $\Gamma_0(N)$ for $N \geq 5$ of the modular group $\mathrm{SL}_2(\mathbb{Z})$  
$$ \Gamma_0(N) = \Bigg\{
\begin{pmatrix}
a & b \\ c & d 
\end{pmatrix}
 \in \mathrm{SL}_2(\mathbb{Z}) \ \ | \ \  c \equiv 0 \mod N \Bigg\}. $$
We denote by $Y_0(N)$ the smooth irreducible affine curve that is the moduli space classifying elliptic curves with $\Gamma_0(N)$ level structure, that is elliptic curves with cyclic subgroups of order $N$. Equivalently, $Y_0(N)$ classifies pairs of elliptic curves related by an $N$-isogeny.
Over $\mathbb{C}$, we have 
$$\mathbb{H}/ \Gamma_0(N) \ \simeq \ Y_0(N)_{\mathbb{C}} \ : \ \tau \ \mapsto \ \left( \mathbb{C}/(\mathbb{Z}+\mathbb{Z}\tau) \ , \ \langle \dfrac{1}{N} \rangle \right).$$
We denote by $X_0(N)$ the compactification of $Y_0(N)$ viewed as a Riemann surface and we let $j$ be the inclusion map $$j: Y_0(N) \hookrightarrow X_0(N).$$ 
The assumption $N\geq 5$ allows for the definition of the universal elliptic curve $$\pi: \mathscr{E} \longrightarrow X_0(N).$$ 
Let $$ \mathbb{Z}^2 \ \backslash \  (\mathbb{C} \times \mathbb{H})$$ be the universal generalized elliptic curve over the Poincar\'{e} upper half plane where $(m,n)$ in $\mathbb{Z}^2$ acts on $\mathbb{C} \times \mathbb{H}$ by 
$$(z, \tau) \ \mapsto \ (z+ m \tau + n, \tau ).$$
We denote by $\mathscr{E}$ the compact universal generalized elliptic curve of level $\Gamma_0(N)$.
Let $W_{r}$ be the \emph{Kuga-Sato variety} of dimension $r+1$, that is a compact desingularization of the $r$-fold fibre product
$$ \mathscr{E} \times_{X_0(N)} \cdots \times_{X_0(N)} \mathscr{E},$$ (see \cite{deligne1973modular} and the appendix by Conrad in \cite{bertolini2013generalized} for more details).

Fix a prime $p$ with $(p,N\phi(N))=1$.
Consider the sheaf $$\mathcal{F}=Sym^{2r-2}(R^1\pi_* \mathbb{Z}/p) .$$ 
Let $$\Gamma_{2r-2}=( \mathbb{Z}/N \rtimes \mu_2)^{2r-2} \rtimes \Sigma_{r-2}$$ where $\mu_2=\{ \pm 1\}$ and $\Sigma_{2r-2}$ is the symmetric group on $2r-2$ elements. Then $\Gamma_{2r-2}$ acts on $W_{2r-2}$, (see \cite[Sections~1.1.0,1.1.1]{scholl1990motives} for more details.)
The projector $$e_r \in \mathbb{Z}\left[\dfrac{1}{2N(r-2)!}\right][\Gamma_{2r-2}]$$ associated to $\Gamma_{2r-2}$, called Scholl's projector, belongs to the group of zero correspondences $$\mathrm{Corr}^0(W_{2r-2},W_{2r-2})_{\mathbb{Q}}$$ from $W_{2r-2}$ to itself over $\mathbb{Q}$, (see \cite[Section~2.1]{bertolini2012chow} for more details.).

\begin{remark}
The hypothesis $((2r-2)!,p)=1$ is not necessary by a combination of the work of Tsuji \cite{tsuji1994padic} on $p$-adic comparison theorems and Saito \cite{saito2000weight} on the Weight-Monodromy conjecture for Kuga-Sato varieties.
\end{remark}

\begin{proposition} \label{propcohomology}

$$H^1_{et}(X_0(N) \otimes \overline{\mathbb{Q}}, j_* \mathcal{F}) \simeq e_{r} \oplus_{i=0}^{r+1} H^i_{et}(W_r  \otimes \overline{H},\mathbb{Z}/p).$$
\end{proposition}

\begin{proof}
The proof is a combination of \cite[theorem~1.2.1]{scholl1990motives} and \cite[proposition~2.4]{bertolini2013generalized}. Note that the proof in \cite[theorem~1.2.1]{scholl1990motives} involves $\mathbb{Q}_p$ coefficients but it is still valid in our setting, (see the Remark following \cite[Proposition~2.1]{nekovar1992kolyvagin}).
\end{proof}

Define $$J=  H^1_{et}(X_0(N) \otimes \overline{\mathbb{Q}}, j_* \mathcal{F}).$$
For primes $\ell$ prime to $N$, the Hecke operators $T_\ell$ act on $X_0(N)$,
which induces an endomorphism of $H^1_{et}(X_0(N) \otimes \overline{\mathbb{Q}}, j_* \mathcal{F})$.
Let $A$ be its $f$-isotypic component with respect to the action of the Hecke operators.
Let $$I = Ker \{ \mathbb{T} \longrightarrow \mathcal{O}_F: T_\ell \longrightarrow a_\ell, \forall \ell \nmid N \}.$$
Then $A= \{ x \in J \ | \ Ix=0 \} $ is isomorphic to $J/IJ$. $A$ is a free $\mathcal{O}_F \otimes \mathbb{Z}_p$ module of rank 2, equipped with a continuous $\mathcal{O}_F$-linear action of $\Gal(\overline{\mathbb{Q}}/\mathbb{Q})$. Hence, there is a map $$e_A: J \longrightarrow A$$ that is equivariant under the action of Hecke operators and $\Gal(\overline{\mathbb{Q}}/\mathbb{Q})$.

Consider the \'{e}tale $p$-adic Abel-Jacobi map 
$$\Phi: CH^{r}(W_{2r-2}/H)_0 \longrightarrow H^1 (H,H^{2r-1}_{et} ( W_{2r-2} \otimes \overline{H},\mathbb{Z}_p (r ) ) )$$ where $CH^{r}(W_{2r-2}/H)_0$ is the group of homologically trivial cycles of codimension $2r-2$ on $W_{2r-2}$ defined over $H$, modulo rational equivalence.
Composing the Abel-Jacobi map with the projector $e_{r}$, we obtain a map $$\Phi: CH^{r}(W_{2r-2}/H)_0 \longrightarrow H^1(H,J).$$
The Abel-Jacobi map commutes with automorphisms of $W_{2r-2}$, so $\Phi$ factors through $$e_{r}(CH^{r}(W_{2r-2}/H)_0 \otimes \mathbb{Z}_p).$$ Proposition \ref{propcohomology} implies that $e_{r}H^{r+1}(W_{2r-2} \otimes \overline{H},\mathbb{Z}_p)=0$. 
Since $$CH^{r}(W_{2r-2}/H)_0= Ker(CH^{r}(W_{2r-2}/H) \longrightarrow H^{r+1}(W_{2r-2} \otimes \overline{H},\mathbb{Z}_p)),$$ we have $e_{r}(CH^{r}(W_{2r-2}/H)_0 \otimes \mathbb{Z}_p)=e_{r}(CH^{r}(W_{2r-2}/H) \otimes \mathbb{Z}_p)$.
Composing the former map with the map $e_A: J \longrightarrow A$, we get $$\Phi: e_{r}CH^{r}(W_{2r-2}/H)_0 \longrightarrow H^1(H,A).$$

\section{Heegner cycles} \label{Heegner cycles}

Consider an integer $m$ such that $(m,cNDp)=1.$
Recall that $H=K_c$ is the ring class field of $K$ of conductor $c$. We denote by $$H_m=K_{cm}$$ the ring class field of $K$ of conductor $cm$ for $m> 1$.
We describe Nekov{\'a}{\v{r}}'s construction of Heegner cycles as in \cite[Section~5]{nekovar1992kolyvagin}.

By the Heegner hypothesis, there is an ideal $\mathcal{N}$ of $\mathcal{O}_{cm}$, the order of $K$ of conductor $cm$, such that $$\mathcal{O}_{cm}/\mathcal{N} = \mathbb{Z}/N \mathbb{Z}.$$ Therefore, $\mathbb{C}/\mathcal{O}_{cm}$ and $\mathbb{C}/\mathcal{N}^{-1}$ define elliptic curves over $\mathbb{C}$ related by an $N$-isogeny. As points of $X_0(N)$ correspond to elliptic curves over $\mathbb{C}$ related by $N$-isogenies, this provides a \emph{Heegner point} $x_{m}$ of $X_0(N)$. By the theory of complex multiplication, $x_m$ is defined over the ring class field $H_m$ of $K$ of conductor $cm$, (see \cite{gross1984heegner} for more details). 
Let $E$ be the elliptic curve corresponding to $x_m$. Then $E$ has complex multiplication by $\mathcal{O}_{cm}$.
Letting $graph(\sqrt{-D})$ be the graph of the multiplication by $\sqrt{-D}$ on $E$, we denote by $Z_{E }$ the image of the divisor $$(graph(\sqrt{-D})-E\times 0 -D(0 \times E))$$ in the N\'{e}ron-Severi group $NS(E \times E)$ of $E \times E$, that is, the group of divisors of $E \times E$ modulo algebraic equivalence.
Consider the inclusion $$i: E^{2r-2} \longrightarrow W_{2r-2}. $$
Then $i_*(Z_E^{r-1})$ belongs to the Chow group $\mathrm{CH}^{r}(W_{2r-2}/H_m)_0$. 
Denote by $y_m$ the image of $i_*(Z_E^{r-1})$ by the $p$-adic \'{e}tale Abel-Jacobi map $$\Phi:\mathrm{CH}^{r}(W_{2r-2}/H_m)_0 \longrightarrow H^1(H_m,A)$$ 
as described in \cite{jannsen1988continous}.
We consider two crucial properties of the Galois cohomology classes thus obtained from Heegner cycles.

A prime $\ell$ inert in $K$ where $(\ell,cmNDp)=1$ is unramified in $H_m$. A prime $\lambda_m$ above $\ell$ in $H_m$ ramifies completely in $H_{\ell m}$.     
\begin{proposition} \label{corestriction}
Consider cohomology classes $y_n$ and $y_m$ with $n=\ell m$, where $\ell$ is a prime inert in $K$. Then $$T_{\ell} y_m= \cor_{H_n/H_m} y_n=a_{\ell} y_m.$$
\end{proposition}

\begin{proof}
Let $E_m$ be the elliptic curve corresponding to $x_m$. Then, we have
$$T_{\ell} (i_*(Z_{E_m}^{r-1}) )= \sum_{y } i_*(Z_{E_y}^{r-1}) ,$$ where the elements $ y  \in Y_0(N)$  correspond to $\ell$-isogenies $E_y \rightarrow E_m$ compatible with level $\Gamma_0(N)$ structure.
The set $\{y\}$ consists of the orbit of $x_n$ in $$\Gal(H_n/H_m) \simeq \Gal(K_n/K_m) \simeq \Gal(K_{\ell}/K_1).$$ 
Let $E_n$ be the elliptic curve corresponding to $ x_n$.
We have $$\sum_{y } i_*(Z_{E_y}^{r-1})= \sum_{g \in \Gal(H_n/H_m)} g \cdot i_*(Z_{E_n}^{r-1})  = \cor_{H_n/H_m} i_*(Z_{E_n}^{r-1}).$$
Since the action of the Hecke operators commutes with the Abel-Jacobi map, we obtain
$$T_{\ell} y_m =\cor_{H_n/H_m}y_n.$$
The equality $$T_{\ell}y_m=a_{\ell} y_m$$ follows from the definition of $A$ on which Hecke operators $T_{\ell}$ act by $a_{\ell}$. 
\end{proof}

We denote by $(y_n)_v$ the image of an element $y_n \in H^1( H_n ,A )$ in $H^1(H_{n,v},A)$. 

\begin{proposition}\label{local relation}
Consider cohomology classes $y_n$ and $y_m$ with $n=\ell m$, where $\ell$ is a prime inert in $K$. Let $\lambda_m$ be a prime above $\ell$ in $K_{m}$ and $\lambda_n$ the prime above $\lambda_m$ in $K_n$.
Then $$(y_n)_{\lambda_n} = Fr(\ell) (\mathrm{res}_{K_{\lambda_m}, K_{\lambda_n}} (y_m)_{\lambda_m} ) \hbox{ in } H^1(K_{\lambda_n},A).$$
\end{proposition}
\begin{proof}
The proof can be found in \cite[proposition~6.1(2)]{nekovar1992kolyvagin}.
\end{proof}

\section{The Euler system} \label{The Euler System}
Let $n=\ell_1 \cdots \ell_k$ be a squarefree product of primes $\ell_i$ inert in $K$ satisfying 
$$(\ell_i,DNpc)=1 \ \hbox{ for } \ i=1, \cdots, k.$$ The Galois group $G_n = \Gal(H_n/H)$ is isomorphic to the product over the primes $\ell$ dividing $n$ of the cyclic groups $\Gal(H_{\ell}/H)$ of order $\ell+1$.
Let $\sigma_{\ell}$ be a generator of $G_{\ell}$.
We denote by $\mathcal{O}_{\wp}$, the completion of $\mathcal{O}_F$ at a prime $\wp$ dividing $p$.
Then $\mathcal{O}_{F} \otimes \mathbb{Z}_p = \oplus_{\wp |p} \mathcal{O}_{\wp}.$
Let 
\begin{equation} \label{localized A}
A_{\wp}= A \otimes_{\mathcal{O}_F \otimes \mathbb{Z}_p} \mathcal{O}_{\wp}
\end{equation}
be the localization of $A$ at $\wp$.
Denote by $$y_{n,\wp} \in H^1(H_n,A_{\wp})$$ the $\wp$-component of $y_n \in H^1(H_n,A)$.
In this section, we use Operators \eqref{Kolyvagin operators0} considered by Kolyvagin to define Kolyvagin cohomology classes $P(n) \in H^1(H,A_{\wp}/p)$ using the cohomology classes $y_n$ in $ H^1(H_n,A)$ for appropriate $n$.
Let $$L= H(A_{\wp}/p)$$ be the smallest Galois extension of $H$ such that $\Gal(\overline{\mathbb{Q}}/L)$ acts trivially on $A_{\wp}/p$. 
We will denote by $\Frob_{F_1/F_2}(\alpha)$, the conjugacy class of the Frobenius substitution of the prime $\alpha$ of $F_2$ in $\Gal(F_1/F_2)$.

A prime $\ell$ will be referred to as a \emph{Kolyvagin prime} if it is such that $$(\ell,DNpc)=1 \hbox{ and } \Frob_{\ell}(L/\mathbb{Q})=\Frob_{\infty}(L/\mathbb{Q}),$$ where $\Frob_{\infty}(L/\mathbb{Q})$ refers to the conjugacy class of complex conjugation. 
Given a Kolyvagin prime $\ell$, the Frobenius condition implies that it is inert in $K$. 
Denote by $\lambda$ the unique prime in $K$ above $\ell$. Since $\lambda$ is unramified in $H$ and has the same image as $\Frob_{\infty}(L/K)=\tau^2=Id$ by the Artin map, it splits completely in $H$. 
Let $\lambda'$ be a prime of $H$ lying above $\lambda$, then $\lambda'$ splits completely in $L$ as it lies in the kernel of the Artin map:
 $$\Frob_{\lambda'}(L/ H)
=\tau^2=Id.$$ 
 
The Frobenius condition also implies that 
\begin{align} \label{congruence}
a_{\ell} \equiv \ell+1 \equiv 0 \mod p.
\end{align}
Indeed, the characteristic polynomial of the complex conjugation acting on $A_{\wp}/p$ is $x^2-1$
while the characteristic polynomial of $\Frob(\ell)$ is $$x^2 -a_{\ell}/\ell^{r} x +1/\ell.$$ The latter corresponds to the polynomial $x^2 -a_{\ell} x +\ell^{2r-1}$ where we make the change of variable $x \rightarrow \ell^rx $ dictated by the Tate twist $r$ of $Y_p$. As a consequence, we obtain the polynomial $$x^2 \ell^{2r} -a_{\ell} \ell^{r} x +\ell^{2r-1}= \ell^{2r}(x^2 -a_{\ell}/\ell^{r} x +1/\ell).$$  

We assume that the primes $\ell$ dividing $n$ are Kolyvagin primes.
Let 
\begin{align} \label{Kolyvagin operators0}
\Tr_{\ell}=\sum_{i=0}^{\ell} \sigma_{\ell}^i, \qquad D_{\ell}= \sum_{i=1}^{\ell} i \sigma_{\ell}^i.
\end{align} 
These operators are related by $$(\sigma_{\ell} -1) D_{\ell} = {\ell}+1-\Tr_{\ell}.$$
We define $D_n= \prod_{{\ell}|n} D_{\ell}$ in $\mathbb{Z}[G_n]$.
And we denote by $\mathrm{red}(x)$ the image of an element $x$ of $ H^1(H_n, A_{\wp})$ in $H^1(H_n, A_{\wp}/p)$ induced by the projection $$A_{\wp} \longrightarrow A_{\wp}/p.$$ 

\begin{proposition} \label{fixed}
We have
$$D_n \mathrm{red}(y_{n,\wp}) \hbox{ belongs to } H^1(H_n,A_{\wp}/p)^{G_n}.$$ 
\end{proposition}
\begin{proof}
It is enough to show that for all $\ell$ dividing $n$, $$(\sigma_{\ell}-1)D_n \mathrm{red}( y_{n,\wp})=0$$ in $H^1(H_n,A_{\wp}/p)$.
We have $$(\sigma_{\ell}-1)D_n=(\sigma_{\ell}-1)D_\ell D_m=(\ell+1-\Tr_{\ell})D_m.$$
Since $\mathrm{res}_{H_m,H_n} \circ \cor_{H_n/H_m} = \Tr_{\ell}$, Proposition \ref{corestriction} implies 
$$(\ell+1-\Tr_{\ell})D_m \mathrm{red}( y_{n,\wp}) = (\ell+1)D_m \mathrm{red} (y_{n,\wp})- a_{\ell} \mathrm{res}_{H_m,H_n} (D_m \mathrm{red}( y_{m,\wp})).$$
The latter is congruent to 0 modulo $ p$ by Equation \eqref{congruence}. 
\end{proof}

\begin{proposition}\label{galois group}
For $n$ such that $(n,cpND)=1$, we have $$H^0(H_n,A_{\wp}/p)=H^0(\mathbb{Q}, A_{\wp}/p)=0,$$ $$ \hbox{and } \Gal(H_n(A_{\wp}/p)/H_n) \simeq \Gal(H(A_{\wp}/p)/H) \simeq \Gal(K(A_{\wp}/p)/K) \simeq \Gal(\mathbb{Q}(A_{\wp}/p)/\mathbb{Q} ).$$
\end{proposition}
\begin{proof}
Indeed, $H_n/\mathbb{Q}$ and $\mathbb{Q}(A_{\wp}/p)/\mathbb{Q}$ are unramified outside primes dividing $cnD$ and $ Np$ respectively, so $H_n \cap \mathbb{Q}(A_{\wp}/p) $ is unramified over $\mathbb{Q}$. Since $\mathbb{Q}$ has no unramified extensions, we obtain that $H_q \cap \mathbb{Q}(A_{\wp}/p)=\mathbb{Q}$, and therefore $H^0(H_q, A_{\wp}/p)=H^0(\mathbb{Q},Y_p)$.
The hypothesis $$ \Gal(\mathbb{Q}(A_{\wp}/p)/\mathbb{Q}) \simeq \mathrm{GL}_2(\mathcal{O}_{\wp}/p)$$ further implies that $H^0(\mathbb{Q},A_{\wp}/p)=0$.
The result follows.
\end{proof}

\begin{proposition}
The restriction map $$\mathrm{res}_{H,H_n}: H^1(H, A_{\wp}/p) \longrightarrow H^1(H_n, A_{\wp}/p)^{G_n}$$ is an isomorphism for $(n,cpND)=1$.
\end{proposition}
\begin{proof}
This follows from the inflation-restriction sequence:
\begin{align*}
0 & \rightarrow H^1(H_n/H,A_{\wp}/p) \xrightarrow{inf} H^1(H,A_{\wp}/p) \xrightarrow{\mathrm{res}} H^1(H_n,A_{\wp}/p)^{G_n} \rightarrow H^2(H_n/H,A_{\wp}/p)
\end{align*}
using the fact that $H^0(H_n, A_{\wp}/p)=0$ by Proposition \ref{galois group}. 
\end{proof}

As a consequence, the cohomology classes $D_n \mathrm{red}(y_{n,\wp})$ can be lifted to cohomology classes $P(n)$ in $H^1(H,A_{\wp}/p)$ such that $$\mathrm{res}_{H,H_n} P(n) =  D_n \mathrm{red}(y_{n,\wp}).$$

\begin{proposition} \label{ramification}
Let $v$ be a prime of $H$. If $v |N$, then $P(n)_v$ is trivial. If $v \nmid Nnp$, then $P(n)_v$ lies in $H^1(H_v^{ur}/H_v,A_{\wp}/p)$.
\end{proposition}
\begin{proof}
If $v$ divides $N$, we follow the proof in \cite[lemma~10.1]{nekovar1992kolyvagin}. 
We denote by $$(A_{\wp}/p)^{dual}=\Hom(A_{\wp}/p,\mathbb{Z}/p\mathbb{Z}(1)  )$$ the local Tate dual of $A_{\wp}/p$.
The local Euler characteristic formula \cite[Section~1.2]{milne1986arithmetic} yields $$| H^1(H_v,A_{\wp}/p)|=|H^0(H_v,A_{\wp}/p)|\times |H^2(H_v,A_{\wp}/p)|.$$
Local Tate duality then implies $$| H^1(H_v,A_{\wp}/p)|=|H^0(H_v,A_{\wp}/p)|^2.$$ 
The Weil conjectures and the assumption on $Fr(v)$ imply that $((A_{\wp}/p)^{I_v})^{Fr(v)}=0$ where $$<Fr(v)> \ =\ \Gal(H_v^{ur}/H_v)$$ and $I_v=\Gal(\overline{H_v}/H_v^{ur} )$ is the inertia group. 
Therefore, $$((A_{\wp}/p)^{I_v})^{G(H_v^{ur}/H_v)} = (A_{\wp}/p)^{G(\overline{H_v}/H_v)} = H^0(H_v,A_{\wp}/p)= 0.$$

To prove the second assertion, if $v$ does not divide $ Nnp$, we observe that $$\mathrm{res}_{H,H_n} P(n)_v=D_nred(y_{n,\wp})_{v'}$$ belongs to $H^1(H^{ur}_{n,v'}/H_{n,v'},A_{\wp}/p)$ and $H_{n,v'}/H_v$ is unramified for $v'$ in $H_n$ above $v$.
\end{proof}

\section{Localization of Kolyvagin classes} \label{Localization0}
Nekov{\'a}{\v{r}} \cite{nekovar1992kolyvagin} studied the relation between the localization of Kolyvagin cohomology classes $P(m \ell)$ and $P(m)$, for appropriate $m$ and $\ell$ by explicitly computing cocycles using the Euler system properties. We briefly explain his development in this section.

\paragraph{Set up.}
We denote by
\begin{align*}
& G_1=\Gal(\overline{\mathbb{Q}}/H_{ 1}), \ \ G_{\ell }=\Gal(\overline{\mathbb{Q}}/H_{\ell }), \ \ \tilde{G_{ 1}}=\Gal(\overline{\mathbb{Q}}/H_{1}^+),\\
\hbox{ and } 
& G_{\lambda_1}=\Gal(\overline{\mathbb{Q}_\ell}/H_{1,\lambda_1}), \ \ G_{\lambda_{\ell }}=\Gal(\overline{\mathbb{Q}_\ell}/H_{\ell ,\lambda_{\ell}} ) ,\ \ \tilde{G_{\lambda_1}}=\Gal(\overline{\mathbb{Q}_\ell}/\mathbb{Q}_\ell), 
\end{align*}
where $H_{ 1}^+$ is the maximal real subfield of $H_{ 1}$.
Then $$G_1/G_{\ell }=<\sigma>, \ \ \tilde{G_1}/G_1=<\tau>, \ \ \tilde{G_1}/G_{\ell } = \Gal(H_{\ell }/H^{+})= <\sigma> \rtimes <\tau>$$ for some $\sigma$ and $\tau$ of order $\ell+1$ and 2 respectively.
There is a surjective homomorphism 
$$\pi: \tilde{G_{\lambda_1}} \xrightarrow[]{res} \Gal(\mathbb{Q}_\ell^{t}/\mathbb{Q}_\ell ) = \hat{\mathbb{Z}}^{'}(1) \rtimes  2 \hat{\mathbb{Z}},$$
where $$\Gal(\mathbb{Q}_\ell^{t}/\mathbb{Q}_\ell^{ur} )  \simeq \hat{\mathbb{Z}}^{'}(1)= \prod_{ q \neq \ell} \mathbb{Z}_{\ell}$$ is generated by an element $\tau_\ell$ and $$\Gal(\mathbb{Q}_\ell^{ur}/\mathbb{Q}_\ell )\simeq  \hat{\mathbb{Z}}$$ is generated by the Frobenius element $\phi$ at $\ell$ and $\phi \tau_\ell \phi^{-1} = (\tau_\ell)^{\ell}$. 
One can show that $$H^1(G_{\lambda_1},A_{\wp}/p)=H^1(G_{\lambda_{\ell }},A_{\wp}/p) \simeq H^1( 2\hat{Z}, A_{\wp}/p) \simeq (A_{\wp}/p)/((\phi^2-1)A_{\wp}/p)$$
and a cocycle $F$ in $Z^1(\ \hat{\mathbb{Z}}^{'}(1) \rtimes 2 \hat{\mathbb{Z}}, \ A_{\wp}/p)$ acts by
$$F(\tau_\ell^u \phi^{2v})=(1+\phi^2+\cdots+\phi^{2(v-1)})a + (\phi^2-1) b,$$ where $[F]=a \mod (\phi^2-1)A_{\wp}/p.$

Let $H_{\lambda}^{ur}$ be the maximal unramified extension of $H_{\lambda}$, and let $H_{\lambda}^t$ be the maximal tamely ramified extension of $H_{\lambda}$, 
We denote by \begin{equation}
\gamma  \ : \ H^1(H_{\lambda}^{ur}/H_{\lambda},A) \simeq H^1(H_{\lambda}^{ur}, A)^{\phi}
\end{equation}
the isomorphism that switches cocycles with same values on the arithmetic Frobenius element $\phi$ generating $\Gal(H_{\lambda}^{ur}/H_{\lambda})$ and the generator $\tau_\ell$ of $\Gal(H_{\lambda}^{t}/H_{\lambda}^{ur} )$ modulo $p$.

\begin{proposition}
We have 
\begin{align} \label{localrelation}
\left(\dfrac{\ell+1}{p}\epsilon - \dfrac{a_\ell}{p} \right) \gamma(P(m)_{\lambda_1})= \dfrac{a_\ell \epsilon/\ell^r -1/ \ell  -1}{p} P(\ell m)_{\lambda_1}
\end{align}
where $\lambda_1$ is a prime of $H_1$ dividing $\ell$, and $\epsilon= \pm 1$.
Furthermore, $P(\ell m)_{\lambda_1}$ is unramified at $\ell$.
\end{proposition}
\begin{proof}
We denote by $$x=D_m y_m \in H^1(G_1,A_{\wp}/p), \ \ \hbox{and} \ \ y=D_m y_{\ell m} \in H^1(G_{\ell },A_{\wp}/p).$$
Let $z=P(\ell m)$ in $ H^1(G_1, A_{\wp}/p)$. Then $$\res_{G_1,G_{\ell }}(z)= D_\ell red(y) \in H^1(G_{\ell },A_{\wp}/p).$$
For $a $ in $A_{\wp}/p$, we have $$D_\ell a = \sum_{i=1}^{\ell} i \sigma^i(a) = \sum_{i=1}^{\ell} i =\dfrac{\ell(\ell+1)}{2} \equiv 0 \mod p.$$
Therefore, $\res_{G_1, G_{\lambda_{\ell }}}(z)=0$, which implies that $P(\ell m)_{\lambda_{1 }}$ is ramified at a place $\lambda_1$ of $H_1$ above $\ell$. Hence, using the inflation-restriction sequence
$$0 \longrightarrow H^1(G_{\lambda_1}/G_{\lambda_{\ell }},A_{\wp}/p) \longrightarrow H^1(G_{\lambda_1},A_{\wp}/p) \longrightarrow H^1(G_{\lambda_{\ell }},A_{\wp}/p) \longrightarrow 0,$$
we obtain
$$P(\ell m)_{\lambda_1} =\res_{G_1,G_{\lambda_1}}(z)= \mathrm{inf}_{G_{\lambda_1}/G_{\lambda_{\ell }}, G_{\lambda_1}}(z_1)$$ for some 
$$z_1 \in H^1(G_{\lambda_1}/G_{\lambda_{\ell }},A_{\wp}/p) = \Hom(<\sigma>, A_{\wp}/p).$$
Since $cor_{G_{\ell },G_1}(y)= a_\ell x $, there is an element $a$ in $ A_{\wp}/p$ such that 
\begin{align} \label{corestriction equation}
\cor_{G_{\ell },G_1}(y)(g_1)-a_\ell \ x(g_1)=(g_1-1)a
\end{align} for $g_1$ in $ G_1$.
It is shown in \cite[section~7]{nekovar1992kolyvagin} that $$a=z_1(\sigma).$$ 
We let $$a_x=\res_{G,G_{\lambda_1}}(x), \ \hbox{ and } \  a_y=\res_{H,G_{\lambda_{\ell }}}(y).$$
Restricting $g_1$ to $g_{\lambda_1} \in G_{\lambda_1}$ in equation \eqref{corestriction equation} where $\pi(g_{\lambda_1})=\sigma^u \phi^{2v}$, we obtain
$$\sum_{i=0}^{\ell} a_y (\tilde{\sigma}^{-i} g_{\lambda_1} \tilde{\sigma}^{i})-a_\ell a_x(g_{\lambda_1})= (\ell+1) a_y (g_{\lambda_1} )-a_\ell a_x(g_{\lambda_1})=(\phi^2-1) a,$$
where $\tilde{\sigma}$ is a lift of $\sigma$ in $ G_1/G_\ell$ to $G_1$.
We have
\begin{align*} 
& x(g_{\lambda_1})=(1+\phi^2+ \cdots + \phi^{2(v-1)})a_x+(\phi^2-1) b_x,\\
 \& \ \  & y(g_{\lambda_1})=(1+\phi^2+ \cdots + \phi^{2(v-1)})a_y+(\phi^2-1) b_y.
\end{align*}
For $u=0, v=1$, we obtain from the last three equations
\begin{align}\label{relationequation}
(\ell +1) y(g_{\lambda_1})-a_\ell x(g_{\lambda_1})=(\phi^2-1)a + (\phi^2-1)(-a_\ell b_x+(\ell+1) b_y),
\end{align}
where $$(\phi^2-1)(-a_\ell b_x+(\ell+1) b_y)=0 \mod p$$ 
as $a_\ell \equiv \ell+1 \equiv 0 \mod p$.
The second property of the Euler system $$a_y=\phi(a_x) \mod (\phi^2-1) \ A_{\wp}/p$$ implies that
$$\dfrac{\ell+1}{p} y(g_{\lambda_1})-\dfrac{a_\ell}{p} x(g_{\lambda_1})=\left(\dfrac{\ell+1}{p}\epsilon - \dfrac{a_\ell}{p} \right) x(g_{\lambda_1})$$
where $\epsilon$ is such that $\phi \equiv \tau$ acts by $\epsilon$ on $a_x$. 
Therefore, by Equation \eqref{relationequation}, $$\left(\dfrac{\ell+1}{p}\epsilon - \dfrac{a_\ell}{p} \right) x(g_{\lambda_1})= \dfrac{\phi^2-1}{p}a \mod p.$$
The characteristic polynomial of $\phi$ implies that
\begin{align*}
\phi^2-a_\ell\phi/\ell^{r}+1/\ell=0 \hbox{ on } A_{\wp}/p. 
\end{align*}
Therefore, $$\left(\dfrac{\ell+1}{p}\epsilon - \dfrac{a_\ell}{p} \right) x(g_{\lambda_1}) =\dfrac{a_\ell \phi/\ell^r -1/ \ell  -1}{p}a \equiv \dfrac{a_\ell \epsilon/\ell^r -1/ \ell  -1}{p}a.$$
We seek to express $a= z_1(\sigma)$ in terms of $P(\ell m)_{\lambda_1}=\mathrm{inf}_{G_{\lambda_1}/G_{\lambda_{\ell }},G_{\lambda_1}}(z_1)$ where the generator $\sigma$ of $G_{\lambda_1}/G_{\lambda_{\ell }}$ can be lifted to the generator $\tau_\ell$ of $G_{\lambda_1}=\Gal(H_{\lambda_1}^t/H_{\lambda_1}^{ur})$. It is therefore enough to apply the map $\gamma$ to $a$ to obtain $P(\ell m)_{\lambda_1}$ where $\gamma$ switches cocycles with same values on $\Frob(\ell)$ and $\tau_\ell$. The result follows.
\end{proof}

\section{Statement}

Recall that the Galois group $G=\Gal(H/K)$ where $H$ is the ring class field of $K$ of conductor $c$ acts on $H^1(H,A_{\wp}/p)$.
We denoted by $$\hat{G}=\Hom(G, \mu_e)$$ the group of characters of $G$
and by $$e_{\chi}=\dfrac{1}{|G|}\sum_{g \in G} \chi^{-1}(g) g$$ the projector onto the $\chi$-eigenspace given a character $\chi$ of $\hat{G}$.
We let $$\delta = \mathrm{red}(y_{1}) \ \hbox{ in } \ H^1(H,A_{\wp}/p).$$ Then $e_{\overline{\chi}}\delta$ belongs to the $\overline{\chi}$-eigenspace of $H^1(H,A_{\wp}/p)$.
We recall the statement of the theorem we prove.
\begin{reptheorem}{theorem 1}
Assume that $p$ is such that 
\begin{align} \label{hyp0}
 \Gal \left( \mathbb{Q}(A_{\wp}/p)/\mathbb{Q} \right) \simeq \mathrm{GL}_2(\mathcal{O}_{\wp}/p), \ \ \ \ (p,ND\phi(N))=1, \ \hbox{ and } \ p \nmid |G|.
\end{align}
Suppose further that the eigenvalues of $Fr(v)$ acting on $A_{\wp}^{I_v}$ are not equal to 1 modulo $p$ for $v$ dividing $N$. 
Assume $\chi \in \hat{G}$ is such that $e_{\overline{\chi}}\delta$ is non-zero.
Then the $\chi$-eigenspace $S^{\chi}$ of the Selmer group  $S$ has rank 1 over $\mathcal{O}_{\wp}/p$.
\end{reptheorem}
\textbf{Set up of the proof.} 
Consider the prime $\lambda$ of $K$ lying above a prime $\ell$ inert in $\mathbb{Q}$ and let $\lambda'$ be a prime of $H$ above $\lambda$. 
The self-duality of $A_{\wp}/p$ given by 
$$A_{\wp}/p \simeq \Hom(A_{\wp}/p,\mathbb{Z}/p\mathbb{Z}(1)  ),$$ where $\Hom(A_{\wp}/p,\mathbb{Z}/p\mathbb{Z}(1)  )$ is the Tate dual of $A_{\wp}/p$ 
and local Tate duality gives a perfect pairing
$$\langle.,.\rangle_{\lambda'}:H^1( H^{ur}_{\lambda'}/H_{\lambda'},(A_{\wp}/p)^{I_{\lambda'}}) \times H^1(H_{\lambda'}^{ur}, A_{\wp}/p) \longrightarrow \mathbb{Z}/p \mathbb{Z},$$
where 
$I_{\lambda'} = \Gal(\overline{H_{\lambda'}}/H_{\lambda'}^{ur})$
and $\mathcal{O}_{\wp}$-linear isomorphisms
\begin{align}\label{eq}
\{ H^1(H_{\lambda'}^{ur},A_{\wp}/p)\}^{dual}  \simeq H^1(H^{ur}_{\lambda'}/H_{\lambda'},(A_{\wp}/p)^{I_{\lambda'}})\simeq (A_{\wp}/p)^{I_{\lambda'}}/(\phi-1).
\end{align}
where $\phi$ is the arithmetic Frobenius element generating $\Gal(H^{ur}_{\lambda'}/H_{\lambda'})$. 
Recall that the Selmer group $S \subseteq H^1(H, A_{\wp}/p)$ consists of the cohomology classes whose localizations lie in $$H^1(H^{ur}_{v}/H_{v},A_{\wp}/p)$$ for $v$ not dividing $Np$ and in $H^1_f(H_v, A_{\wp}/p)$ for $v$ dividing $p$. Here, $H^1_f(H_v, A_{\wp}/p)$ is the \emph{finite part} of $H^1(H_v, A_{\wp}/p)$ as in \cite{block1990lfunction}.
We denote by $$\mathrm{res}_{\lambda}: H^1(H,A_{\wp}/p) \longrightarrow \oplus_{\lambda'|\lambda} H^1(H_{\lambda'},A_{\wp}/p)$$ the direct sum of the restriction maps from $H^1(H,A_{\wp}/p)$ to $H^1(H_{\lambda'},A_{\wp}/p)$ for $\lambda'$ dividing $\lambda$ in $H$.
Restricting $\mathrm{res}_{\lambda}$ to the Selmer group, we obtain the following map
$$\mathrm{res}_{\lambda}: S\longrightarrow \oplus_{\lambda'|\lambda} H^1(H^{ur}_{\lambda'}/H_{\lambda'},(A_{\wp}/p)^{I_{\lambda'}}).$$
Taking the $(\mathbb{Z}/p) $-linear dual of the previous map and using isomorphism \eqref{eq}, we obtain a homorphism 
$$\psi_\ell:  \oplus_{\lambda'|\lambda} H^1(H_{\lambda'}^{ur},A_{\wp}/p) \longrightarrow S^{dual}.$$
Let $$X_{\ell}=\mathrm{Im}(\psi_\ell)$$ be the image of $\psi_{\ell}$ in $S^{dual}$. 
We aim to bound $S^{dual}$ from above by using the Kolyvagin classes $P(n)$ introduced in Section \ref{The Euler System} to produce explicit elements in the kernel of $\psi_\ell$.

\section{Generating the dual of the Selmer group}
\begin{lemma} \label{Sah1} 
We have $$ H^1(Aut(A_{\wp}/p), A_{\wp}/p) = 0.$$
\end{lemma}
\begin{proof}
Sah's lemma \cite[8.8.1]{lang1983fundamentals} states that if $G$ is a group, $M$ a $G$-representation, and $g$ an element of $\mathrm{Center}(G)$, then the map $x \longrightarrow (g-1) \ x$ is the zero map on $H^1(G,M)$.
In our context, since $$g = 2 I \in \Aut(A_{\wp}/p) $$ belongs to $ \mathrm{Center}(\Aut(A_{\wp}/p)),$
we have that $g-I = I$ is the zero map on the group  $$H^1(Aut(A_{\wp}/p), A_{\wp}/p)$$ and the result follows.
\end{proof}

\begin{proposition}
There exists a prime $q$ such that $q$ is a Kolyvagin prime, and such that $$\mathrm{res}_{\beta'}e_{\overline{\chi}} \delta$$ is not divisible by $p$ where $\beta'$ is a prime dividing $q$ in $H$.
\end{proposition} 
\begin{proof}
For the purpose of this proof, we denote the cocycle $e_{\overline{\chi}} \delta$ by $c_1$ and the Galois group $G(L/H)$ by $G$. By Proposition \ref{ramification}, $c_1$ belongs to $S^{\overline{\chi} } $.
The restriction map $$r: H^1(H,A_{\wp}/p) \longrightarrow H^1(L,A_{\wp}/p)^G=\Hom_G(L,A_{\wp}/p)$$ is injective.
Indeed, Proposition \ref{galois group} and Proposition \ref{Sah1}
imply that 
$$ \Ker(r)= H^1(H(A_{\wp}/p)/H, A_{\wp}/p)=0.$$
Consider the evaluation pairing 
$$r(S^{\overline{\chi} }) \times \Gal(\overline{\mathbb{Q}}/L) \longrightarrow A_{\wp}/p$$ and let $$\Gal_S(\overline{\mathbb{Q}}/L)$$ be the annihilator of $r(S^{\overline{\chi} })$.
Let $L^S$ be the extension of $L$ fixed by $\Gal_S(\overline{\mathbb{Q}}/L)$ and denote by $G_S$ the Galois group $\Gal(L^S/L)$. We obtain an injective homomorphism of $\Gal(H/\mathbb{Q})$-modules
$$r(S^{\overline{\chi} }) \hookrightarrow \Hom_G(G_S, A_{\wp}/p ).$$
We denote by $s$ the image of $r(c_1)$ in $\Hom_G(G_S, A_{\wp}/p )$. 

If $s(G_S^+)=0$, then as $s$ belongs to $S^{ \pm }$, we have $$s:G_S^- \longrightarrow A_{\wp}/p^{\pm},$$ where $A_{\wp}/p^{\pm}$ are the $\pm$ eigenspaces of $A_{\wp}/p$ with respect to the action of $\tau$. On the one hand, the eigenspace $A_{\wp}/p^{\pm}$ is of rank one over $\mathcal{O}_{\wp}/p$. On the other hand, by Proposition \ref{galois group} and Assumption \eqref{hyp0}, $$G=G(L/H) \simeq \mathrm{GL}_2(\mathcal{O}_{\wp}/p).$$ Hence, $A_{\wp}/p^{\pm}$ has no non-trivial $G$-submodules and $s(G_S^-)=0$, that is $s=0$. This is a contradiction because $ c_1 \neq 0$ in $S^{\overline{\chi} }$ as $c_1$ is not divisible by $p$ in $S^{\overline{\chi} }$.
As a consequence, we have that $s(G_S^+) \neq 0$, where $$G_S^+ =G_S^{\tau+1}= \{ h^{\tau}h \ |\ h \hbox{ in } G_S \}= \{ (\tau h)^2 \ | \ h \hbox{ in } G_S \}.$$ Therefore, there exists $h$ in $G_S$ such that $c_1((\tau h)^2) \neq 0$.
Consider the element $\tau h $ in $\Gal(L^S/\mathbb{Q})$. Cebotarev's density theorem implies the existence of $q$ in $\mathbb{Q}$ such that $$\Frob_q(L^S/\mathbb{Q})=\tau h $$ and such that $(q,cpND)=1$. In particular, $q$ is a Kolyvagin prime since $\mathrm{res}|_L (\tau h) = \tau$.
For $\beta$ in $L$ above $q$, we have that $$\Frob_{\beta}(L^S/L) = (\tau h)^2$$
generates the local extension $L^S/L$ at $ \beta$. This implies that $\mathrm{res}_{\beta'} c_1$ does not vanish for $$\beta'=\beta \cap H.$$ 
\end{proof}

We consider the restriction $d$ of an element $c$ of $H^1(H, A_{\wp}/p )$ to $H^1(F, A_{\wp}/p )$. Then $d$ factors through some finite extension $\tilde{F}$ of $F$. We denote by $$F(c)=\tilde{F}^{\ker(d)}$$ the subextension of $\tilde{F}$ fixed by $\ker(d)$. Note that $F(c)$ is an extension of $F$. 

Consider the following extensions
\begin{displaymath}
    \xymatrix{  & I_{01}=I_0I_1 & \\
	        I_0=F(\mathrm{red}(y_{1,\wp}))^{\Gal} \ar[ur] &  & I_1= F(D_q \mathrm{red}(y_{q, \wp}))^{\Gal} \ar[ul] \\
                & F=H_{q}(A_{\wp}/p)\ar[ur]_{V_1} \ar[ul]^{V_0} \ar@{-->}[uu]^{V} & }
\end{displaymath}
where the abbreviation $\Gal$ indicates taking Galois closure over $\mathbb{Q}$.
We define $$V_0=\Gal(I_0/F) \hbox{, } V_1=\Gal(I_1/F) \hbox{, and }V=\Gal(I_0I_1/F).$$
We have an isomorphism of $\Aut(A_{\wp}/p)$-modules $V_0 \simeq V_1 \simeq A_{\wp}/p$.
Let $$I_0^{\overline{\chi}}=F(e_{\overline{\chi}}\mathrm{red}(y_{1,\wp}))^{\Gal} \hbox{ and } I_1^{\overline{\chi}}=F(e_{\overline{\chi}}D_qred( y_{q, \wp}))^{\Gal}.$$ We denote by $V_0^{\overline{\chi}}$ and $V_1^{\overline{\chi}}$ their respective Galois groups over $F$.
We will show that $$V^{\overline{\chi}}=\Gal(I_0^{\overline{\chi}}I_1^{\overline{\chi}}/F) \simeq V_0^{\overline{\chi}} \times V_1^{\overline{\chi}}.$$

\begin{proposition}\label{disjoint}
The extensions $I_0^{\overline{\chi}}$ and $I_1^{\overline{\chi}}$ are linearly disjoint over $F$.
\end{proposition}
\begin{proof}
Linearly independent cocycles $c_1,c_2$ of $H^1(H_{q},A_{\wp}/p) $ over $\mathcal{O}_{\wp}/p$ can be viewed as linearly independent homomorphisms $h_1,h_2$ in $\Hom_{\Gal(F/H_{q})} (V,A_{\wp}/p)$ over $\mathcal{O}_{\wp}/p$. 
The restriction map $$ H^1(H_{q},A_{\wp}/p)^{\Gal(F/H_{q})} \xrightarrow{(r)} H^1(F,A_{\wp}/p)^{\Gal(F/H_{q})}$$ is injective. Indeed, combining Proposition \ref{galois group} with
Proposition \ref{Sah1}
that implies that $$ H^1(K(A_{\wp}/p)/K, A_{\wp}/p)=0,$$ we obtain that
$$\Ker(r)= H^1(F/H_q, A_{\wp}/p)=0.$$
Furthermore, cocycles of $H^1(F,A_{\wp}/p)^{\Gal(F/H_{q})}$ factor through $$ H^1(I_{01}/F,A_{\wp}/p)^{\Gal(F/H_{q})}=\Hom_{\Gal(F/H_{q})}(I_{01}/F,A_{\wp}/p).$$
Consider the extension $I_0^{\overline{\chi}} \cap I_1^{\overline{\chi}}$ of $F$. It is a $\Gal(F/H_{q})$-submodule of $A_{\wp}/p$.  
The hypothesis $ \mathrm{res}_{\beta'} e_{\overline{\chi}} \mathrm{red}(y_{1,\wp})\neq 0 $ implies that $$\mathrm{res}_{\beta'}  e_{\overline{\chi}} \mathrm{red} (D_q y_{q,\wp}) \neq 0$$ by \eqref{local relation}. On the one hand, since $\mathrm{res}_{\beta'}  e_{\overline{\chi}} \mathrm{red} (D_q y_{q,\wp})$ is ramified, $e_{\overline{\chi}}D_q \mathrm{red}(y_{q, \wp})$ does not belong to $S^{\overline{\chi}}$. On the other hand, $ e_{\overline{\chi}} \mathrm{red}(y_{1,\wp}) \neq 0$ belongs to $S^{\overline{\chi}}$ by Proposition \ref{ramification}. Therefore $I_0^{\overline{\chi}} \cap I_1^{\overline{\chi}}=0$ since $A_{\wp}/p$ is a simple $\Gal(F/H_{q})$-module.
Note that the cocycles $c_1$ and $c_2$ cannot be linearly dependent either since one of them belongs to $S^{\overline{\chi}}$ while the other one does not. 
\end{proof}

For a subset $U \subseteq V$, we denote by $$L(U)=\{ \ell \hbox{ rational prime } | \Frob_{\ell}(I_{01}/\mathbb{Q})=[\tau u ], u \in U \}.$$
Note that a rational prime $\ell$ in $L(U)$ is a Kolyvagin prime as $$\Frob_{\ell}(H(A_{\wp}/p)/ \mathbb{Q})=\mathrm{res}|_{H(A_{\wp}/p)} \Frob_{\ell}(I_{01}/\mathbb{Q}) = \tau$$ since $u \in U$. In fact, it satisfies $$\Frob_{\ell}(H_q(A_{\wp}/p)/ \mathbb{Q})=\mathrm{res}|_{H_q(A_{\wp}/p)}\Frob_{\ell}(I_{01}/\mathbb{Q})=\tau.$$
Hence, a prime above ${\ell}$ in $H$ splits completely in $H_q$. Indeed, it lies in the kernel of the Artin map because of the Frobenius condition $$\Frob_{\ell}(H_q/ H)=\tau^{|D(H/\mathbb{Q})|}=\tau^2=Id,$$ where $|D(H/\mathbb{Q})|$ is the order of the decomposition group $D(H/\mathbb{Q})$, also the order of the residue extension.
Similarly, a prime above ${\ell}$ in $H_q$ splits completely in $H_q(A_{\wp}/p)$; it lies in the kernel of the Artin map because of the Frobenius condition $$\Frob_{\ell}(H_q(A_{\wp}/p)/ H_q)=\tau^{|D(H_q/\mathbb{Q})|}=\tau^2=Id.$$

\begin{proposition}\label{SDual}
Assume $U^{+}$ generates $V^{+}$. Then $\{ X_\ell \}_{\ell \in L(U)}$ generates $S^{dual}$.
\end{proposition} 
\begin{proof} 
The proof consists of the following steps:
\begin{enumerate}[leftmargin=0cm,itemindent=.5cm,labelwidth=\itemindent,labelsep=0cm,align=left]
\item An element $s$ of $S$ can be identified with an element $h$ of $\Hom_{G}(F, A_{\wp}/p)$.
\item To show the statement of the theorem, it is enough to show that $\mathrm{res}_{\lambda}( s)=0$ for all $ \ell \in L(U)$ implies $s=0$.
\item The assumption $\mathrm{res}_{\lambda}( s)=0$ for all $ \ell \in L(U)$ implies that $h$ vanishes on $U^+$.
\item The assumption $U^{+}$ generates $V^{+}$ implies $h=s=0$.
\end{enumerate}
\begin{enumerate}[leftmargin=0cm,itemindent=.5cm,labelwidth=\itemindent,labelsep=0cm,align=left]
\item
Let $s$ be an element of $ S$. For the purpose of this proof, we denote $$G=\Gal(H(A_{\wp}/p)/H)\simeq \mathrm{GL}_2(\mathcal{O}_{\wp}/p) .$$
We denote by $h$ the image of $s$ by restriction in $$H^1(F,A_{\wp}/p)^G \subset \Hom_{G}(\Gal(\overline{F}/F), A_{\wp}/p).$$ Here, restriction can be viewed as the composition of the following two restriction maps
$$H^1(H, A_{\wp}/p) \xrightarrow{(r_1)} H^1(H(A_{\wp}/p),A_{\wp}/p)^{G} \xrightarrow{(r_2)} H^1(F,A_{\wp}/p)^{G}.$$
Combining Proposition \ref{galois group} and \ref{Sah1} 
we obtain that
$$ \Ker(r_1)= H^1(H(A_{\wp}/p)/H, A_{\wp}/p)=0.$$
By Proposition \ref{galois group}, we have $$\Gal(H_q(A_{\wp}/p)/H(A_{\wp}/p))\simeq \Gal(H_q/H) \simeq \mathbb{Z}/(q+1)\mathbb{Z}.$$ 
On the one hand, the group $G$ acts trivially on $\Gal(H_q(A_{\wp}/p)/H(A_{\wp}/p))$. On the other hand, $A_{\wp}/p$ is simple as a $G$-module. Hence,
$$ \Ker(r_2)=  \Hom_{G}(\Gal(F/H(A_{\wp}/p)),A_{\wp}/p)\simeq \Hom_{G}(H_q/H,A_{\wp}/p)=0$$
since such a $G$-homomorphism maps an element of $\Gal(H_q/H)$ to a $G$-invariant element of $A_{\wp}/p$, that is, to 0.
\item
By Isomorphism \eqref{eq}, local Tate duality identifies $\oplus_{\lambda'|\lambda} H^1(H_{\lambda'}^{ur}, A_{\wp}/p)$ with $$\oplus_{\lambda'|\lambda}
 H^1(H^{ur}_{\lambda'}/H_{\lambda'},(A_{\wp}/p)^{I_{\lambda'}}).$$ So if we show that $$\{ \mathrm{res}_{\lambda} \}_{\ell \in L(U) }: S\longrightarrow \{ \oplus_{\lambda'|\lambda} H^1(H_{\lambda'}^{ur}/H_{\lambda'}, (A_{\wp}/p)^{I_{\lambda'}}) \}_{ \ell \in L(U) }$$ is injective, then
the induced map between the duals $$\{ \oplus_{ \lambda'|\lambda} H^1(H_{\lambda'}^{ur}, A_{\wp}/p) \}_{\ell \in L(U) } \longrightarrow  S^{dual}$$ would be surjective. Hence, it is enough to show that $\mathrm{res}_{\lambda}( s)=0$ for all $ \ell \in L(U)$ implies $s=0$.
\item
Consider $\tilde{I_{01}}$, the minimal Galois extension of $\mathbb{Q}$ containing $I_{01}$ such that $h$ factors through $\Gal(\tilde{I_{01}}/ F)$. Let $x$ be an element of $\Gal(\tilde{I_{01}}/F)$ such that $x|_{I_{01}}$ belongs to $ U$.
By Cebotarev's density theorem, there exists $\ell$ in $L(U)$ such that $\Frob_\ell(\tilde{I_{01}}/ \mathbb{Q})=[\tau x ].$
The hypothesis $\mathrm{res}_{\lambda}(s)=0$ implies that $h(\Frob_{\lambda''}(\tilde{I_{01}}/F))=0$ for $\lambda''$ above $\ell$ in $F$ since $\Frob_{\lambda''}(\tilde{I_{01}}/F)$ is a generator of the local extension of $\Gal(\tilde{I_{01}}/F)$ at $\lambda''$. 
In fact,
$$\Frob_{\lambda''}(\tilde{I_{01}}/F)= (\tau x )^{|D(F/\mathbb{Q})| }=(\tau x)^2=x^{\tau}x=2x^{+},$$
where $|D(F/\mathbb{Q})|$ is the order of the decomposition group $D(F/\mathbb{Q})$, and is also the order of the residue extension and $x^{+}=\dfrac{1}{2} x^{\tau}x.$
Therefore, $h(x^{+})=0$ for all $x \in \Gal(\tilde{I_{01}}/F)$ such that $x|_{I_{01}}$ belongs to $U$.
\item
The hypothesis $U^{+}$ generates $V^{+}$ then implies that $h$ vanishes on $\Gal(\tilde{I_{01}}/F)^{+}.$
Hence, $\mathrm{Im}(h)$ lies in $A_{\wp}/p^{-}$, the minus eigenspace of $A_{\wp}/p$ for the action of $\tau$ which is a free $\mathcal{O}_{\wp}/p$-module of rank 1.  
In particular, it cannot be a proper non-trivial $G$-submodule of $A_{\wp}/p$.
Therefore, $h=0$ which implies $s=0$.
\end{enumerate}
\end{proof}

Next, we study the action of complex conjugation on the $\chi$-component of the cocycles $ y_{q,\wp}$.
\begin{proposition} \label{conjugation}
There is an element $\sigma_0$ in $\Gal(H_{q}/K)$ such that $$\tau e_{\chi} y_{q,\wp}= \epsilon \overline{\chi}(\sigma_0) e_{\overline{\chi}}y_{q,\wp}, $$ where $- \epsilon$ is the sign of the functional equation of $L(f,s)$.
\end{proposition}
\begin{proof}
\cite[proposition~6.2]{nekovar1992kolyvagin} that uses a result in \cite{gross1984heegner} states that 
\begin{align} \label{eq1}
\tau y_{q,\wp}=  \epsilon \sigma_0 y_{q,\wp} 
\end{align}
for some $\sigma_0$ in $\Gal(H_{q}/K)$.
Since $\tau$ acts on an element $g$ of $G$ by $$\tau g \tau^{-1}=g^{-1},$$
we have $$\tau e_{\chi}= \dfrac{1}{|G|}\sum_{g \in G} \tau \chi^{-1}(g) g= \dfrac{1}{|G|}\sum_{g \in G} \chi (g^{-1}) g^{-1} \tau = \dfrac{1}{|G|}\sum_{g \in G} \overline{\chi}^{-1}(g^{-1}) g^{-1} \tau = e_{\overline{\chi}} \tau .$$
Also, $$ e_{\overline{\chi}}\sigma_0 =  \dfrac{1}{|G|}\sum_{g \in G}  \overline{\chi}^{-1}(g)\sigma_0 g=
\dfrac{1}{|G|}\sum_{g \in G}\overline{\chi}(\sigma_0)  \overline{\chi}^{-1}(\sigma_0 g)\sigma_0 g=\overline{\chi}(\sigma_0) e_{\overline{\chi}}.$$
Therefore, applying $e_{\overline{\chi}}$ to Equation \eqref{eq1} yields $$\tau e_{\chi} y_{q,\wp}= \epsilon \overline{\chi}(\sigma_0) e_{\overline{\chi}}y_{q,\wp} .$$

\end{proof}

Let us look at the action of complex conjugation on $V^{\overline{\chi}}=V_0^{\overline{\chi}} V_1^{\overline{\chi}}$.
For $(v_0,v_1)$ in $ V_0  V_1$, we use the identity $\tau D_q =- D_q \tau \mod p$ to obtain 
\begin{align*}
& \tau v_0 \tau (e_{\overline{\chi}}y_{1,\wp})= \epsilon \chi (\sigma_0) \tau  v_0(e_{\chi} y_{1,\wp}).\\
& \tau  v_1 \tau ( e_{\overline{\chi}}D_{q}y_{q,\wp})= - \tau  v_1 D_{q} \tau (e_{\overline{\chi}}y_{q,\wp}) = - \epsilon \chi (\sigma_0) \tau  v_1(e_{\chi} D_q y_{q,\wp}).  
\end{align*}
When $\chi = \overline{\chi}$, for $(x,y)$ in $V_0^{\overline{\chi}} V_1^{\overline{\chi}}$,  $$\tau (x,y) \tau = (\epsilon \chi(\sigma_0) \tau x, - \epsilon \chi (\sigma_0) \tau y).$$
In this case, we define $$U= \{ (x,y) \hbox{ in } V_0\times V_1 | \epsilon \overline{\chi}(\sigma_0) \tau x + x, - \epsilon \overline{\chi} (\sigma_0) \tau y + y \hbox{ generate } A_{\wp}/p \} .$$
When $\chi \neq \overline{\chi}$, for $(x,y,z,w)$ in $V_0^{\chi} V_0^{\overline{\chi}}V_1^{\chi} V_1^{\overline{\chi}}=V,$
$$\tau (x,y,z,w) \tau = (\epsilon \overline{\chi}(\sigma_0) \tau y, \epsilon \chi(\sigma_0) \tau x, - \epsilon \overline{\chi}(\sigma_0) \tau w,-\epsilon \chi(\sigma_0) \tau z).$$
In this case, we define $$U= \{ (x,y,z,w) \hbox{ in }V_0^{\chi} V_0^{\overline{\chi}}V_1^{\chi} V_1^{\overline{\chi}} | \epsilon \chi (\sigma_0) \tau x + y, - \epsilon \overline{\chi} (\sigma_0) \tau z + w
\hbox{ generate } A_{\wp}/p \} .$$
In both cases, Proposition \ref{disjoint} and Congruence \eqref{congruence} imply that $U^{+}$ generates $$V^{+} \simeq V_0^+ \times V_1^+ \simeq \mathcal{O}_{\wp}/p \times \mathcal{O}_{\wp}/p \simeq A_{\wp}/p.$$
Let $\ell$ be a prime in $L(U)$, and let $\lambda$ be the prime of $K$ lying above it.

\begin{proposition} \label{generators}
The elements $$\mathrm{res}_{\lambda} e_{\overline{\chi}} P(\ell) \hbox{ and } \mathrm{res}_{\lambda} e_{\overline{\chi}} P(\ell q) $$ generate $\oplus_{\lambda'|\lambda} H^1(H_{\lambda'}^{ur},A_{\wp}/p)^{\overline{\chi}}$.
\end{proposition}
\begin{proof} 
We have $$\oplus_{\lambda'|\lambda} H^1(H_{\lambda'}^{ur}, A_{\wp}/p)^{\overline{\chi}} \simeq \oplus_{\lambda'|\lambda} \left((A_{\wp}/p)^{I_{\lambda'}}/(\phi-1)\right)^{\overline{\chi}}$$ since the former is isomorphic to its dual by Isomorphism \eqref{eq}. The module $$\oplus_{\lambda'|\lambda} \left((A_{\wp}/p)^{I_{\lambda'}}/(\phi-1)\right)^{\overline{\chi}}$$ is of rank at most 2 over $\mathcal{O}_{\wp}/p$, hence, so is $\oplus_{\lambda'|\lambda} H^1(H_{\lambda'}^{ur}, A_{\wp}/p)^{\overline{\chi}}$.
The Frobenius condition on $\ell$ implies that $$\mathrm{res}_{\lambda} e_{\overline{\chi}} \mathrm{red}(y_{1, \wp}) \hbox{ and } \mathrm{res}_{\lambda} e_{\overline{\chi}} D_q \mathrm{red}( y_{q, \wp}) $$ are linearly independent over $ \oplus_{\lambda'|\lambda}A_{\wp}/p$. 
Indeed, if they were linearly dependent then, in the case $\chi=\overline{\chi}$,
\begin{equation*}
\begin{split}
& (\mathrm{res}_{\lambda} e_{\overline{\chi}} \mathrm{red}(y_{1, \wp}))^{(\tau x)^2}-\mathrm{res}_{\lambda} e_{\overline{\chi}} \mathrm{red}(y_{1, \wp}) \\
 \hbox{ and } & (\mathrm{res}_{\lambda} e_{\overline{\chi}} D_qred( y_{q, \wp}) )^{(\tau y)^2}-\mathrm{res}_{\lambda} e_{\overline{\chi}} D_q \mathrm{red}(y_{q, \wp})
\end{split}
\end{equation*}
where $\Frob_{\ell}(I_{01}/\mathbb{Q})=  \tau u= (\tau x,\tau y)$
would also be linearly dependent. The Frobenius condition implies that 
\begin{align*}
\Frob_{\ell}(I_0^{\overline{\chi}}/F)= x^{\tau }  x=(\tau x)^2 \hbox{ and }\Frob_{\ell}(I_1^{\overline{\chi}}/F)= y^{\tau } y= (\tau y)^2
\end{align*}
generate $A_{\wp}/p$, which yields a contradiction as $(\tau x)^2$ acts on the element $\mathrm{res}_{\lambda} e_{\overline{\chi}} \mathrm{red}(y_{1, \wp})$ generating the local extension of $I_0^{\overline{\chi}}$ over $F$ by $$(\mathrm{res}_{\lambda} e_{\overline{\chi}} \mathrm{red}(y_{1, \wp}))^{(\tau x)^2}-\mathrm{res}_{\lambda} e_{\overline{\chi}} \mathrm{red}(y_{1, \wp})$$
and $(\tau y)^2$ acts on the element $\mathrm{res}_{\lambda} e_{\overline{\chi}} D_qred( y_{q, \wp})  $ generating the local extension of $I_1^{\overline{\chi}}$ over $F$ by $$\mathrm{res}_{\lambda} e_{\overline{\chi}} D_qred( y_{q, \wp}) )^{(\tau y)^2}-\mathrm{res}_{\lambda} e_{\overline{\chi}} D_q \mathrm{red}(y_{q, \wp}).$$ 
Similarly, in the case $\chi \neq \overline{\chi}$, 
\begin{equation*}
\begin{split} 
&(\mathrm{res}_{\lambda}e_{\overline{\chi}} \mathrm{red}(y_{1, \wp}))^{ x^{\tau }   y}-\mathrm{res}_{\lambda}e_{\overline{\chi}} \mathrm{red}(y_{1, \wp}) \\
\hbox{ and } & (\mathrm{res}_{\lambda} e_{\overline{\chi}}D_q \mathrm{red}( y_{q, \wp}) )^{z^{\tau }   w}-\mathrm{res}_{\lambda}e_{\overline{\chi}}D_q \mathrm{red}( y_{q, \wp})
\end{split}
\end{equation*}
where $\Frob_{\ell}(I_{01}/\mathbb{Q})=  \tau u= (\tau x,\tau y,\tau z,\tau w)$
would also be linearly dependent. The Frobenius condition implies that $$\Frob_{\ell}(I_0^{\overline{\chi}}/F)= x^{\tau }  y=(\tau x) (\tau y) \hbox{ and } \Frob_{\ell}(I_1^{\overline{\chi}}/F)= z^{\tau }   w= (\tau z)(\tau w)$$ generate $A_{\wp}/p$, which yields a contradiction.

Equation \eqref{localrelation} implies that 
if $\mathrm{res}_{\lambda}e_{\overline{\chi}} P(\ell q)$ and  $\mathrm{res}_{\lambda} e_{\overline{\chi}}P(\ell)$ were linearly dependent then $$\mathrm{res}_{\lambda}e_{\overline{\chi}} P(q)=\mathrm{res}_{\lambda} e_{\overline{\chi}}D_q \mathrm{red} y_{q, \wp} \hbox{ and } \mathrm{res}_{\lambda} e_{\overline{\chi}}P(1)=\mathrm{res}_{\lambda}e_{\overline{\chi}} \mathrm{red} y_{1,\wp}$$ would be linearly dependent as well.
\end{proof}

\section{Bounding the size of the dual of the Selmer group}
In what follows, we study the modules $X_\ell^{\overline{\chi}}$ for $\ell$ in $L(U)$.

\begin{proposition} \label{reciprocity}
We have $$\sum_{\lambda'|\ell|n} \langle s_{\lambda'},\mathrm{res}_{\lambda'} P(n)\rangle_{\lambda'} = 0.$$
\end{proposition}
\begin{proof}
The proof follows \cite[proposition~11.2(2)]{nekovar1992kolyvagin} where both the reciprocity law and the local ramification properties of $P(n)$ in Proposition \ref{ramification} are used. 
\end{proof}

\begin{proposition}\label{vanishing}
The element $ \psi_\ell(\mathrm{res}_{\lambda} e_{\overline{\chi}} P(\ell q))$ generates $ X_{\ell}^{\overline{\chi}}$ over $O_{\wp}/p$ for $\ell$ in $L(U)$.
\end{proposition}
\begin{proof}
The image of $\mathrm{res}_{\lambda}  e_{\overline{\chi}} P(\ell) $ by the map $$\psi_{\ell}: \oplus_{\lambda'|\lambda} H^1(H_{\lambda'}^{ur}, A_{\wp}/p)^{\overline{\chi}} \longrightarrow X_{\ell}^{\overline{\chi}}$$ is the homomorphism from $S^{\overline{\chi}}$ to $\mathbb{Z}/p$ given by:
$$ e_{\overline{\chi}}s \mapsto \sum_{\lambda'|\lambda} 
\langle e_{\overline{\chi}} s_{\lambda'}, e_{\overline{\chi}} P(\ell)_{\lambda'} \rangle_{\lambda'}.$$
Proposition \ref{reciprocity} implies that 
$$\sum_{\lambda'|\lambda} \langle e_{\overline{\chi}} s_{\lambda'}, e_{\overline{\chi}} P(\ell)_{\lambda'} \rangle_{\lambda'}=0.$$ Hence, the image by $\psi_{\ell}$ of $\mathrm{res}_{\lambda}  e_{\overline{\chi}} P(\ell)$, one of the two generators of $$\oplus_{\lambda'|\lambda} H^1(H_{\lambda'}^{ur}, A_{\wp}/p)^{\overline{\chi}}$$ by Proposition \ref{generators}, is trivial. 
\end{proof}

\begin{proposition}
The modules $ X_{\ell}^{\overline{\chi}}$ that are non-zero are all equal for $\ell \in L(U)$.
\end{proposition}
\begin{proof}
Proposition \ref{reciprocity} implies that $$ \sum_{\lambda'|\lambda} 
\langle e_{\overline{\chi}} s_{\lambda'}, e_{\overline{\chi}} P(\ell q)_{\lambda'} \rangle+ \sum_{\beta'|\beta} 
\langle e_{\overline{\chi}} s_{\beta'}, e_{\overline{\chi}} P(\ell q)_{\beta'} \rangle=0.$$
Hence, $$ \psi_{\ell}( \mathrm{res}_{\lambda}e_{\overline{\chi}} P(\ell q))+ \psi_q(\mathrm{res}_{\beta}e_{\overline{\chi}} P(\ell q)) =0.$$
If $ \psi_{\ell}(\mathrm{res}_{\lambda}  e_{\overline{\chi}} P(\ell q) )=0,$ then by Proposition \ref{vanishing}, $ X_{\ell}^{\overline{\chi}}=0$. 
Otherwise, since $$ \psi_{\ell}(\mathrm{res}_{\lambda}  e_{\overline{\chi}} P(\ell q) )$$ generates $ X_{\ell}^{\overline{\chi}}$ over $\mathcal{O}_{\wp}/ p$, we have that $$- \psi_{\ell}(\mathrm{res}_{\lambda}  e_{\overline{\chi}} P(\ell q) )= \psi_q(\mathrm{res}_{\beta} e_{\overline{\chi}} P(\ell q)) \in X_q^{\overline{\chi}}$$ is non-zero. Therefore, the non-trivial element $\psi_q(\mathrm{res}_{\beta} e_{\overline{\chi}} P(\ell q))$ generates a rank 1 module $ X_{q}^{\overline{\chi}}$ over $\mathcal{O}_{\wp}/ p$ and $ X_{\ell}^{\overline{\chi}}= X_{q}^{\overline{\chi}}$. 
\end{proof}

In what follows, we prove theorem \ref{theorem 1}.
\begin{proof}
By Proposition \ref{SDual}, the set $\{ X_{\ell}^{\overline{\chi}} \}$ generates $S^{dual, \overline{\chi}}$ as $\ell$ ranges over $L(U)$.
Hence, the set $\{  X_{\ell}^{\overline{\chi}} \}$ generates $ S^{dual, \overline{\chi}}$ as $\ell$ ranges over $L(U)$, where, by Proposition \ref{vanishing}, the modules $ X_{\ell}^{\overline{\chi}}$ that are non-zero are of rank 1 over $\mathcal{O}_{\wp}/ p$ and are all equal. Hence, $ \rank (  S^{\chi}) \leq 1$.
Also, $e_{\chi}\mathrm{red}(y_{1, \wp}) $ belongs to $S^{\chi}$ by Proposition \ref{ramification} and is not divisible by $p$ in $S^{\chi}$. Indeed, this follows from the hypothesis on $e_{\overline{\chi}}\mathrm{red}(y_{1, \wp}) $ and Proposition \ref{conjugation} where $\overline{\chi}(\sigma_0)$ is a root of unity since $\Gal(H/K)$ is a finite group. This implies that $\rank(   S^{\chi})\geq 1$.
Therefore,
$$\rank(  S^{\chi}) = \rank(   S^{dual, \overline{\chi}})=1.$$
\end{proof}

\begin{remark}
Because the $p$-adic Abel-Jacobi map factors through the Selmer group, (see \cite[Proposition~11.2.1]{nekovar1992kolyvagin} for a proof) 
$$ \Phi^{\chi}: \mathrm{CH}^r(W_{2r-2}/H)^{\chi}_0 \otimes  \mathcal{O}_{\wp}/p \mathcal{O}_{\wp} \longrightarrow S^{\chi},$$ Theorem \ref{theorem 1} implies that $\rank_{ \mathcal{O}_{\wp}/p}(\mathrm{Im}(\Phi^{\chi}))=1$.
\end{remark}
\begin{remark}
In Kolyvagin's argument for elliptic curves $E$ over $\mathbb{Q}$ and certain imaginary quadratic fields $K$, the non- triviality of the Heegner point $y_K$ in $E(K)/p E(K)$ for suitable primes $p$ immediately implied the non-triviality of $y_K$ in $\Sel_p(E/K)$. In our situation, even though the $p$-adic Abel-Jacobi map is conjectured to be injective, it is non-trivial to check whether a non-trivial Heegner cycle in the Chow group has non-trivial image in $H^1(H,A_{\wp}/p)$.
\end{remark}

\end{document}